\documentclass[twoside,a4paper,reqno,11pt]{amsart}
\usepackage[top=28mm,right=30mm,bottom=28mm,left=30mm]{geometry}
\usepackage{amsfonts, amsmath, amssymb, mathrsfs, bm, latexsym, stmaryrd, array,  mathtools, bold-extra, xcolor}

\usepackage[colorlinks,linkcolor=cyan,anchorcolor=blue,citecolor=red,backref=page]{hyperref}
\usepackage{caption}
\usepackage{listings}
\usepackage{arydshln}
\usepackage[all]{xy}
\usepackage{appendix}

\usepackage{enumitem}
\setlist[enumerate]{font={\rm},itemsep=0.2\baselineskip}
\setlist[enumerate,1]{label={(\roman*)}}
\setlist[enumerate,2]{label={(\arabic*)}}


\newtheorem{theorem}{Theorem}[section]

\newtheorem{proposition}[theorem]{Proposition}

\newtheorem{lemma}[theorem]{Lemma}

\theoremstyle{definition}

\newtheorem{example}[theorem]{Example}

\newtheorem*{conjecture*}{Conjecture}

\newtheorem*{problem*}{Problem}

\theoremstyle{remark}

\def\GL{{\mathrm{GL}}}
\def\GammaL{{\Gamma \mathrm{L}}}
\def\AGL{{\mathrm{AGL}}}

\def\SL{{\mathrm{SL}}}
\def\PSL{{\mathrm{PSL}}}

\def\PGammaL{{\mathrm{P}\Gamma \mathrm{L}}}
\def\PSigmaL{{\mathrm{P}\Sigma \mathrm{L}}}
\def\Sp{{\mathrm{Sp}}}

\def\SU{{\mathrm{SU}}}
\def\PSU{{\mathrm{PSU}}}

\def\CGammaSp{{\mathrm{C}\Gamma\mathrm{Sp}}}

\def\PGammaU{{\mathrm{P}\Gamma\mathrm{U}}}
\def\SigmaL{{\Sigma\mathrm{L}}}

\def\S{{\mathrm{S}}}
\def\A{{\mathrm{A}}}

\def\Q{{\mathrm{Q}}}

\def\H{{\mathrm{H}}}
\def\M{{\mathrm{M}}}

\def\bbZ{{\mathbb{Z}}}
\def\bbF{{\mathbb{F}}}

\def\Sym{{\mathrm{Sym}}}

\def\bfN{{\mathbf{N}}}
\def\bfZ{{\mathbf{Z}}}
\def\bfC{{\mathbf{C}}}

\def\Hom{{\mathrm{Hom}}}
\def\Aut{{\mathrm{Aut}}}

\def\Hol{{\mathrm{Hol}}}

\def\leqs{\leqslant}
\def\geqs{\geqslant}

\def\calP{{\mathcal{P}}}

\def\Magma{{\sc Magma}}

\def\l{\langle}
\def\r{\rangle}
\def\a{\alpha}

\def\Pa{{\sf P}}

\def\rmC{{\mathrm{C}}}

\title{Finite semiprimitive permutation groups of rank $3$}

\author{Cai Heng Li}
\address{SUSTech International Center for Mathematics, and Department of Mathematics, Southern University of Science and Technology, Shenzhen, Guangdong, China}
\email{lich@sustech.edu.cn {\text{\rm(Li)}}}

\author{Hanyue Yi}
\address{Department of Mathematics, Southern University of Science and Technology, Shenzhen, Guangdong, China}
\email{12431017@mail.sustech.edu.cn {\text{\rm(Yi)}}}

\author{Yan Zhou Zhu}
\address{Department of Mathematics, Southern University of Science and Technology, Shenzhen, Guangdong, China}
\email{zhuyz@mail.sustech.edu.cn {\text{\rm(Zhu)}}}

\date{}
\begin{document}

\begin{abstract}
	A transitive permutation group is called \textit{semiprimitive} if each of its normal subgroups is either semiregular or transitive.
	The class of semiprimitive groups properly includes primitive groups, quasiprimitive groups and innately transitive groups.
	The latter three classes of rank $3$ permutation groups have been classified, making significant progress towards solving the long-standing problem of classifying permutation groups of rank $3$.
    In this paper, we complete the classification of finite semiprimitive groups of rank $3$, building on the recent work of Huang, Li and Zhu.
    Examples include Schur coverings of certain almost simple $2$-transitive groups and three exceptional small groups.
\end{abstract}

\maketitle

\section{Introduction}

The \textit{rank} of a transitive permutation group $X\leqs\Sym(\Omega)$ is defined as the number of orbits of $X$ acting on $\Omega\times \Omega$, which equals the number of $X_\a$-orbits on $\Omega$, where $\a\in\Omega$.
The study of finite rank $3$ groups dates back to work of D.~Higman~\cite{higman1964Finite} in the 1960s.
Indeed some important classes of rank $3$ permutation groups have been classified and well-characterized, refer to~\cite{bannai1971Maximal,foulser1969Solvable,kantor1979Permutation,kantor1982Rank,liebeck1987affine,liebeck1986finite} for the primitive case, \cite{devillers2011imprimitive} for the quasiprimitive case,~\cite{baykalov2023Rank} for the innately transitive case.
(Recall that a permutation group is called \textit{innately transitive} if it has a transitive minimal normal subgroup.)

These significant results have many important applications in various combinatorial objects, including partial linear spaces \cite{bamberg2021Partial,devillers2005classification,devillers2008classification} and combinatorial designs~\cite{biliotti2015designs,dempwolff2001Primitive,dempwolff2004Affine,dempwolff2006Symmetric}.

Let $X$ be a transitive permutation group on $\Omega$.
We say $X$ is \textit{semiprimitive} if it is non-regular and every normal subgroup is either transitive or semiregular.
If $X$ is innately transitive, then $X$ has a transitive minimal normal subgroup $N$.
Consequently, innately transitive groups are semiprimitive as centralizers of transitive normal subgroups are semiregular.
However, the converse is not true: a semiprimitive group is not necessarily innately transitive.
In particular, each solvable innately transitive group is primitive, and solvable semiprimitive groups are not necessarily primitive.
For example, the finite Coxeter group $2^3{:}\S_4$ can be viewed as a solvable semiprimitive permutation group on $96$ points, see~\cite[page 1737]{bereczky2008groups}.

The concept of semiprimitive groups was first introduced by Bereczky and Mar{\'o}ti~\cite{bereczky2008groups}, motivated by an application of permutation groups to collapsing transformation monoids.
The class of semiprimitive permutation groups has received considerable attention~\cite{devillers2019distinguishing,giudici2018theory,morgan2020Bounds}, and has been applied to the graph-restrictive groups in~\cite{potocnik2012graphrestrictive}.
Recently, a characterization of semiprimitive groups of rank $3$ is given in~\cite{huang2025finite}, and in this paper, we obtain a complete classification of such groups in Theorem~\ref{thm:semiprim}.

To state the theorem, we need to introduce a few notations.
Let $p$ be a prime.
We recall that a prime $r$ is called a \textit{primitive prime divisor} of $p^k-1$ if $r\mid (p^k-1)$ and $r\nmid(p^j-1)$ for any $j<k$.
Denote by $\mathrm{pdd}(p^k-1)$ the set of primitive prime divisors of $p^k-1$.
Note that $\PGammaL_d(q)=\PSL_d(q).\langle\delta,\phi\rangle$, where $\delta$ is a diagonal automorphism and $\phi$ is a field automorphism of $\PSL_d(q)$. 
Let $\PSigmaL_d(q)=\PSL_d(q).\langle\phi\rangle$.
Moreover, each subgroup $Y\leqslant \PGammaL_d(q)$ containing $\PSL_d(q)$ acts $2$-transitively on the set $\calP$ of $1$-spaces of $\bbF_q^d$.
Denote by $\Pa_1[Y]$ a maximal parabolic subgroup of $Y$ fixing a $1$-space.

\begin{theorem}\label{thm:semiprim}
    Let $X$ be a transitive permutation group on $\Omega$, and let $\alpha\in\Omega$.
    Assume that $X$ is not innately transitive, and $N$ is a non-trivial intransitive normal subgroup of $X$.
    Then $X$ is a semiprimitive rank $3$ group if and only if one of the following cases holds.
    \begin{enumerate}
        \item $N\cong\rmC_r$, $X=N.(\PSL_d(q).\mathcal{O})$ is non-solvable and $X_\alpha\cong\Pa_1[\PSL_d(q).\mathcal{O}]$, where
        \begin{enumerate}
            \item $N.\PSL_d(q)\unlhd X$ is a quotient of $\SL_d(q)$ and is a non-split extension;
            \item $X/N\cong \PSL_d(q).\mathcal{O}\leqslant\PGammaL_d(q)$, and $X/N\not\leqslant\PSigmaL_2(q)$ when $(d,r)=2$;
            \item $q=p^f$ with prime $p$, $r\in\mathrm{ppd}(p^{r-1}-1)$ and $X/\bfC_X(N)\cong\rmC_{r-1}$.
        \end{enumerate}
        \item The tuple $(X,X_\a,\Omega,N)$, with the number of representations, is as in Table~$\ref{tab:semi}$.
    \end{enumerate}
\end{theorem}

\begin{table}[ht]
    \centering
    \caption{Exceptional semiprimitive permutation groups of rank $3$}\label{tab:semi}
    \begin{tabular}{llllcc}
    \hline
        $X$ & $X_\alpha$& $|\Omega|$ & $N$ & \#reps & Ref.\\
        \hline
        $3.\S_6$    & $\S_5$    & $18$  & $3$ & $2$ & \ref{ex:AS} \\
        $2.\M_{12}$ & $\M_{11}$ & $24$  & $2$ & $2$ & \ref{ex:AS} \\
        $\GL_2(3)$  & $\S_3$    & $8$   & $2$ & $2$ & \ref{ex:Q8} \\
        $2^{2+4}{:}\GammaL_1(2^4)$ & $\GammaL_1(2^4)$ & $2^6$ & $2^2$ &$1$ & \ref{ex:Stab-SU3(4)} \\
        $2^4{:}\GL_3(2)$ & $\GL_3(2)$ & $2^4$ & $2$ & $1$ & \ref{exam:two} \\
        \hline  	
    \end{tabular}
\end{table}

We sketch the proof of Theorem~\ref{thm:semiprim} below.
In~\cite[Theorem 1]{huang2025finite}, the classification problem of semiprimitive rank $3$ groups is reduced to the problem of determining certain subgroups of holomorphs of some special $p$-groups $P$.
When $P$ is elementary abelian, we will show in Lemma~\ref{lem:semiaffine} that such a semiprimitive rank $3$ group is exactly the group in the last row of Table~\ref{tab:semi}.
If $P$ is non-abelian, then $\Aut(P)$ has exactly $3$ orbits acting on $P$, which have been classified in~\cite{glasby2024Classifying,li2025finite}.
We will analyse all possible $P$ and $G\leqs\Aut(P)$ in Section~\ref{sec:semi}, and prove that such groups lie in rows~3 and~4 of Table~\ref{tab:semi}.

The layout of the paper is as follows.
In Section~\ref{sec:examples}, we give a brief summary for all finite semiprimitive rank $3$ groups which are not innately transitive.
Theorem~\ref{thm:semiprim} will be proved in Section~\ref{sec:semi}.

\subsection*{Notations}
We denote the cyclic group of order $n$ by $\rmC_n$, or simply by $n$ if there is no ambiguity.
For a group $G$, we write $G_p$ for a Sylow $p$-subgroup of $G$, and $\bfZ(G)$ for the center of $G$.
We also use $G^{(\infty)}$ for the unique perfect group appearing in the derived series of $G$.
For a prime $p$, we use $O_p(G)$ to denote the largest normal $p$-subgroup of $G$, and we use $p^{m+n}$ to denote a non-abelian special $p$-group $N$ such that $|\bfZ(N)| = p^m$ and $N/\bfZ(N) \cong p^n$.
Suppose $G$ acts on a set $\Omega$ and $\Delta\subseteq \Omega$. 
Then we write $G_{\Delta}$ for the setwise stabilizer of $\Delta$ in $G$.
The induced permutation group of $G$ on $\Omega$ is denoted by $G^\Omega$.

\subsection*{Acknowledgments}
The authors are grateful to the anonymous referees for their valuable comments and suggestions that have helped to improve the paper.

\section{Explicit constructions for semiprimitive rank $3$ groups}\label{sec:examples}

In this section, we introduce each semiprimitive group of rank $3$ listed in Theorem~\ref{thm:semiprim}.

Note that $X=\GL_n(3)$ acts transitively on non-zero vectors of $V=\bbF_3^n$, and the stabilizer $X_{v}$ with $v\in V\setminus\{0\}$ has three orbits:
$\{v\}$, $\{2v\}$ and $V\setminus\{0,v,2v\}$.
This yields the following simple example.

\begin{example}\label{ex:Linear-1}
    Let $X=\GL_n(3)$, and let $\Omega$ be the set of non-zero vectors of $\bbF_3^n$.
    Then the natural permutation representation of $X$ on $\Omega$ is semiprimitive of rank $3$.
    \qed
\end{example}

Groups in part~(i) of Theorem~\ref{thm:semiprim} are generalizations of Example~\ref{ex:Linear-1}, which are first introduced in Section~4 of~\cite{huang2025finite}.
These groups come from $r$-fold coverings of $\PGammaL_{n}(q)$ with suitable relations between $r$, $n$ and $q$.

\begin{example}\label{exam:psl}
    Let $p$ be a prime, and let $r\in\mathrm{ppd}(p^{r-1}-1)$.
    Set $V=\bbF_q^d$ with $d\geqslant 2$ and $q=p^f$ such that $f$ is divisible by $r-1$.
    Assume further that $\GammaL_d(q)$ is non-solvable.
    For $C\cong\rmC_{(q-1)/r}<\bfZ(\GL(V))$, let $\Omega$ be the set of $C$-orbits of $V\setminus\{0\}$.
    Then it is proved by~\cite[Proposition 4.5]{huang2025finite} that $\GammaL_d(q)^\Omega\cong \GammaL_d(q)/C$ has rank $3$.
    In particular, for $v\in V\setminus\{0\}$, the $2$ orbits of the stabilizer of $\GammaL_d(q)^\Omega$ on $\Omega\setminus \{v^C\}$ (where $v^C$ is the $C$-orbit containing $v$) are
    \[\bigl\{(\lambda^{t} v)^C:1\leqs t\leqs r-1\bigr\}\mbox{ and }\bigl\{w^C:w\notin \{\lambda^k v:k\in\bbZ\}\bigr\},\]
    where $\lambda$ is a primitive element of $\bbF_q^\times$.
    Moreover, the non-trivial intransitive normal subgroup of $\GammaL_d(q)^\Omega$ is $\bfZ(\GL_d(q))^\Omega$, which is semiregular on $\Omega$.
    This implies that $\GammaL_d(q)^\Omega$ is a semiprimitive group of rank $3$.
    \qed
\end{example}

Using the definitions in Example~\ref{exam:psl}, let $G$ be a group such that $\SL_d(q)\unlhd G\leqs\GammaL_d(q)$.
In~\cite[Proposition 4.11]{huang2025finite}, the conditions under which $G^\Omega$ is a semiprimitive group (but not innately transitive) of rank $3$ are given.
We reformulate them in the following proposition, as in part~(i) of Theorem~\ref{thm:semiprim}.

\begin{proposition}\label{prop:psl}
    Using definitions above, let $N=\bfZ(\GL_d(q))/ C\cong\rmC_r$ and let $X=G^\Omega$.
    The following statements hold.
    \begin{enumerate}
        \item If $X$ is not innately transitive, then $N\lhd X^{(\infty)}\cong r.\PSL_d(q)$.
        \item Assume that $N\lhd X$. 
        Then $X$ is semiprimitive of rank $3$ if and only if $X/\bfC_X(N)\cong\rmC_{r-1}$ and $G\not\leqslant\SigmaL_2(q)$ when $(d,r)=(2,2)$.
    \end{enumerate}
\end{proposition}
\begin{proof}
    Since $G^{(\infty)}=\SL_d(q)$, we have that $X^{(\infty)}=\SL_d(q)/(C\cap \SL_d(q))$.
    Recall that $N=\bfZ(\GL_d(q))/C\cong\rmC_r$.
    Hence, either $N=\bfZ(X^{(\infty)})$ or $X^{(\infty)}\cong\PSL_d(q)$.
    Suppose that $N\not\leqslant X^{(\infty)}$.
    Then $X^{(\infty)}$ is simple, and thus it is a minimal normal subgroup of $X$.
    Recall that $X^{(\infty)}=\left(G^{(\infty)}\right)^\Omega$.
    It follows that $X$ is innately transitive as $G^{(\infty)}\cong\SL_d(q)$ is transitive on $\Omega$.
    Hence, part~(i) holds.

    By~\cite[Lemma 4.7]{huang2025finite}, if $X$ has rank $3$, then $X$ is semiprimitive.
    In addition, if $(d,r)=(2,2)$ then $X$ has rank $3$ if and only if $G\not\leqslant\SigmaL_2(q)$ by \cite[Proposition 4.10]{huang2025finite}.
    Thus, we only need to show that if $(d,r)\neq (2,2)$ then $X$ has rank $3$ if and only if $X/\bfC_X(N)\cong\rmC_{r-1}$.
    
    Note that $N$ is an intransitive normal subgroup of $X$.
    Then an orbit $B$ of $N$ forms a block of $X$ acting on $\Omega$.
    Let $K$ be the kernel of $X_B$ acting on $B$.
    On the one hand, assume that $X$ has rank $3$.
    By \cite[Page 147]{higman1964Finite}, $X_B/K$ is a $2$-transitive group on $B$.
    Hence, $X_B/K\cong \AGL_1(r)$ as $N$ acts regularly on $B$.
    It follows that $X/\bfC_X(N)\cong\Aut(N)\cong \rmC_{r-1}$.
    On the other hand, assume that $X/\bfC_X(N)\cong\rmC_{r-1}$.
    Note that $\bfC_X(N)$ is transitive on $\Omega$ as it contains $(\SL_d(q))^\Omega$.
    Then
    \[\rmC_{r-1}\cong X/\bfC_X(N)=(X_B\bfC_X(N))/\bfC_X(N)\cong X_B/(\bfC_X(N)\cap X_B).\]
    Recall that $X_B/K\lesssim \AGL_1(r)$ and $K\leqslant \bfC_X(N)$.
    This yields that $X_B/K\cong \AGL_1(r)$ is $2$-transitive on $B$.
    Let $\alpha\in B$.
    Then $X_\alpha$ has orbits $\{\alpha\}$ and $B\setminus\{\alpha\}$ on $B$.
    By~\cite[Lemma 4.8]{huang2025finite}, we have that $X_\alpha$ is transitive on $\Omega\setminus B$.
    Thus, $X$ has rank $3$, and therefore the proof is complete.
\end{proof}

Two additional examples arise from certain coverings of almost simple $2$-transitive permutation groups.
These groups are given in~\cite[Theorem 1(a)]{huang2025finite}, which lie in rows~1 and~2 of Table~\ref{tab:semi}.

\begin{example}\label{ex:AS}
	Recall that $\S_6$ has two non-conjugate subgroups which are isomorphic to $\S_5$, one transitive of degree $6$, and the other not; and $\M_{12}$ has two non-conjugate subgroups which are isomorphic to $\M_{11}$, one transitive of degree $12$, and the other not.
	\begin{enumerate}
		\item Let $X=3.\S_6$, a non-split extension.
		Then $X$ has two non-conjugate subgroups which are isomorphic to $\S_5$, and hence $X$ has $2$ inequivalent permutation representations of degree $18$.
		Both of them are semiprimitive of rank $3$.
		\item Let $X=2.\M_{12}$, the Schur extension.
		Then $X$ has two non-conjugate subgroups which are isomorphic to $\M_{11}$ and give rise to $2$ inequivalent permutation representations of degree $24$.
		Both of them are semiprimitive of rank $3$.
		\qed
	\end{enumerate}
\end{example}

The other groups in the last three rows of Table~\ref{tab:semi} come from holomorphs of certain groups.
We remark that $\GL_2(3)$ (row~3 of Table~\ref{tab:semi}) is a subgroup of $\Hol(\Q_8)$.
\begin{example}\label{ex:Q8}
    The quaternion group $\Q_8$ contains exactly $1$ involution and $6$ elements of order $4$.
    Then $\Aut(\Q_8)\cong \S_4$ has a subgroup isomorphic to $\S_3$ which partitions the $8$ elements of $\Q_8$ into $3$ orbits:
    \begin{enumerate}
        \item the single set containing the identity,
        \item the single set containing the unique involution, and 
        \item the set of the $6$ elements of order $4$.
    \end{enumerate}
    Thus, the group $X:=\Q_8{:}\S_3<\Hol(\Q_8)$ has a transitive permutation representation of degree $8$, which has stabilizer $X_\alpha=\l a\r{:}\l b\r=\S_3$.
    Let $c$ be the unique involution of $\Q_8$.
    Then $H=\l a\r{:}\l bc\r\cong \S_3$, which is not conjugate to $X_\alpha$.
    Hence, $X$ has another permutation representation of degree $8$ with stabilizer $H$.
    We remark that both of these two permutation representations are semiprimitive with the unique non-trivial intransitive normal subgroup $N=\bfZ(\Q_8)$, and $X=\Q_8{:}\S_3\cong\GL_2(3)$.
    \qed
\end{example}

The following group arises from $\PGammaU_{3}(2^2)$, which lies in row~4 of Table~\ref{tab:semi}.
\begin{example}\label{ex:Stab-SU3(4)}
    Let $X$ be the stabilizer of $\PGammaU_{3}(2^2)$ acting on $1$-spaces of $\bbF_{2^4}^3$.
    Then
    \[X\cong 2^{2+4}{:}\GammaL_1(2^4)=2^{2+4}{:}(15{:}4)\]
    has a unique minimal normal subgroup $N=2^2$.
    Write $\GammaL_1(2^4)=\l a\r{:}\l\phi\r$.
    Then $\bfN_X(\l a\r)=\l a\r{:}\l\phi\r=\GammaL_1(2^4)$.
    It follows that $X$ has a unique transitive permutation representation with stabilizer isomorphic to $\GammaL_1(2^4)$.
    We remark that $\GammaL_1(2^4)$ acting on $2^{2+4}$ partitions the $2^6$ elements into $3$ orbits:
    \[\{1\},\ N\setminus\{1\}\mbox{ and }2^{2+4}\setminus N.\]
    Thus, this permutation representation is semiprimitive of rank $3$ with the unique non-trivial intransitive normal subgroup $N$.
    \qed
\end{example}

We remark that $\Q_8\cong \SU_3(2)_2$, and hence the semiprimitive groups given in Examples~\ref{ex:Q8} and~\ref{ex:Stab-SU3(4)} have a regular normal subgroup isomorphic to $\SU_3(2^k)_2$ for $k=1$ and $2$, respectively.
For a small group $N$, we can exhaust all subgroups of $\Aut(N)\lesssim \Hol(N)$ with exactly $3$ orbits on $N$ using~\Magma~\cite{bosma1997Magma}.
The calculations show that semiprimitive rank $3$ subgroups of $\Hol(\SU_3(2^k)_2)$ are isomorphic to the groups constructed in Examples~\ref{ex:Q8} and \ref{ex:Stab-SU3(4)}.

\begin{lemma}\label{lem:unique}
    Let $N_k=\SU_3(2^k)_2$ for $k=1$ and $2$.
    \begin{enumerate}
        \item $X\leqslant \Hol(N_1)$ is semiprimitive of rank $3$ if and only if $X$ is one of the groups in Example~$\ref{ex:Q8}$;
        \item $X\leqslant \Hol(N_2)$ is semiprimitive of rank $3$ if and only if $X$ is the group defined in Example~$\ref{ex:Stab-SU3(4)}$.
    \end{enumerate}
\end{lemma}

The next example, lying in the last row of Table~\ref{tab:semi}, is constructed from the exceptional $2$-transitive action of $\A_8$ of degree $15$, with point stabilizer $2^3{:}\GL_3(2)$.
This construction was first given in~\cite[Example 6.3]{huang2025finite}, see also~\cite[Construction 2.1]{li2024Finitea}.
\begin{example}\label{exam:two}
    Let $H\cong \A_7$ be the subgroup of $\GL(V)\cong \A_8$ with $V=\bbF_2^4$.
    Then $H$ is $2$-transitive on non-zero vectors of $V$.
    Define $X=V{:}G$ with $G=H_v\cong\GL_3(2)$ for some non-zero vector $v\in V$.
\end{example}
\begin{lemma}\label{lem:semiaffine}
    Assume that $X$ is a semiprimitive rank $3$ group with an abelian regular normal subgroup.
    Then either $X$ is primitive, or $X$ is as given in Example~$\ref{exam:two}$.
\end{lemma}
\begin{proof}
    Assume that $X$ is defined as in Example~\ref{exam:two}.
    Note that $X\leqs\Hol(V)\cong \AGL_4(2)$ naturally acts on $V$.
    Since $\A_7$ is $2$-transitive, it follows that $G$ is transitive on $V\setminus\{0,v\}$.
    Hence, $X=V{:}G\cong 2^4{:}\GL_3(2)$ has the following three suborbits
    \[\{0\},\{v\},V\setminus\{0,v\}.\]
    Then the non-trivial normal subgroups of $X$ contained in $V$ are precisely $\langle v\rangle$ and $V$.
    Suppose that $N\not\leqslant V$ is a proper normal subgroup of $X$.
    Then $N\langle v\rangle/\langle v\rangle$ is a normal subgroup of $X/\langle v\rangle\cong\AGL_3(2)$, and hence $N\langle v\rangle=X$.
    This yields that $N\cap\langle v\rangle=1$ and $|X:N|=2$. 
    Since $N\cap V\lhd X$, we have that $N\cap V=1$, which contradicts $|X:N|=2$.
    We conclude that non-trivial proper normal subgroups of $X$ are precisely $\langle v\rangle$ and $V$.
    Hence, $X$ is semiprimitive as $\langle v\rangle$ is semiregular and $V$ is transitive.

    On the other hand, we assume that $X=V{:}G$ is semiprimitive of rank $3$ for some vector space $V$ and $G\leqs\GL(V)$.
    We may further assume that $X$ is imprimitive.
    Then $G$ is reducible on $V$, and $G$ stabilizes a proper non-zero subspace $W$ of $V$.
    Let $K$ be the kernel of $G$ acting on $V/W$.
    Then $W{:}K$ is a normal subgroup of $X$ since $[V,K]\leqs W$.
    Since $X$ is semiprimitive and $W{:}K$ is intransitive, we have that $K=1$.
    Note that $G$ has the following three orbits on $V$:
    \[\{0\}, W\setminus\{0\}\mbox{ and }V\setminus W.\]
    Hence, we have that the induced linear groups $G^W$ and $G^{V/W}$ are both transitive on the non-zero vectors.
    Then $O_p(G)\leqs K$ by~\cite[Lemma 3.3]{li2024Finitea}.
    Thus, it follows that $O_p(G)=1$.
    By \cite[Theorem 1.2]{li2024Finitea}, $V=\bbF_2^4$ and $G\cong\GL_3(2)$.
    Note that $G$ acts faithfully on $V/W$, it follows that $\dim V/W=3$.
    Then $X$ is as given in Example~\ref{exam:two} by~\cite[Lemma 2.3]{li2024Finitea}.
\end{proof}

\section{Semiprimitive rank $3$ groups with a regular normal subgroup}\label{sec:semi}

In this section, we assume that $X$ is a semiprimitive group on $\Omega$ of rank $3$ with a regular normal non-abelian special $p$-group $P$.
We may identify $\Omega$ with $P$, and then $P$ consists of right translations.
Let $G$ be the stabilizer in $X$ of the identity of $P$.
Then $G\leqs\Aut(P)$.

As $\bfZ(P)<P$ is a non-trivial characteristic subgroup of $P$, it is a union of some orbits of $G$ acting on $P$.
Hence, the $3$ orbits of $G$ on $P$ have to be 
\[\mbox{\{$1$\}, $\bfZ(P)\setminus \{1\}$\mbox{ and }$P\setminus\bfZ(P)$.}\]

We give some primary properties of $X$.

\begin{lemma}\label{lem:semi}
    Using notations and definitions above, the following statements hold.
    \begin{enumerate}
        \item $G$ acts faithfully on $P/\bfZ(P)$.
        \item The induced action of $G$ on non-identities of $\bfZ(P)$ (or $P/\bfZ(P)$) is transitive.
        \item $|G{:}\bfC_G(x)|=(|P/\bfZ(P)|-1)|\bfZ(P)|$ for $x\in P\setminus\bfZ(P)$.
        \item Denote by $V$ and $M$ the corresponding $\bbF_pG$-modules of $P/\bfZ(P)$ and $\bfZ(P)$, respectively.
        Then $M$ is isomorphic to a quotient of $\Lambda^2(V)$, the exterior square of $V$, as $\bbF_pG$-modules.
    \end{enumerate}
\end{lemma}
\begin{proof}
    Let $K$ be the kernel acting on $P/\bfZ(P)$.
    Then $\langle \bfZ(P),K\rangle=\bfZ(P){:}K$ is a normal subgroup of $X$.
    Since $X$ is semiprimitive, the groups $\Phi(P)$ and $\Phi(P){:}K$ are semiregular normal subgroups of $X$.
    Therefore, we have that $K=1$, that is, $G$ acts faithfully on $P/\Phi(P)$, as in part~(i).

    Since $\bfZ(P)\setminus\{1\}$ is an orbit of $G$, it is clear that $G^{\bfZ(P)}$ is transitive on non-identities.
    Note that $P\setminus\bfZ(P)$ is also an orbit of $G$.
    It yields that $G^{P/\bfZ(P)}$ is also transitive on non-identities, as in part~(ii).

    Note that $\bfC_G(x)$ is the stabilizer of $G$ acting on $P\setminus\bfZ(P)$.
    Hence, we have that $|G{:}\bfC_G(x)|=|P\setminus\bfZ(P)|=(|P/\bfZ(P)|-1)|\bfZ(P)|$, as in part~(iii).

    Part~(iv) is given in~\cite[Lemma 2.8\,(iv)]{li2025finite}.
\end{proof}

First, we deal with the case that $X=P{:}G$ is solvable.
\begin{lemma}\label{lem:semi-nonsol}
    If $G$ is solvable, then $(X,G)\cong (\GL_2(3),\S_3)$ or $(2^{2+4}{:}\GammaL_1(2^4),\GammaL_1(2^4))$ as given in Examples~$\ref{ex:Q8}$ and~$\ref{ex:Stab-SU3(4)}$, respectively.
\end{lemma}
\begin{proof}
    By~\cite{dornhoff1970imprimitive} (also see~\cite[Theorem 2.16]{li2025finite}), if $G$ is solvable with $3$ orbits on $P$, then $P$ is one of the following groups:
    \begin{enumerate}
        \item[(a)] the extraspecial $3$-group of plus type: $3_{+}^{1+4}$;
        \item[(b)] $P(\epsilon)\cong 2^{3+6}$, see~\cite[Definition 4.5]{li2024Proof};
        \item[(c)] Suzuki $2$-group of type A: $A_2(n,\theta)$, see~\cite[Definition 4.1]{li2024Proof};
        \item[(d)] Sylow $p$-subgroup of $\SU_3(p^n)$ for some prime $p$.
    \end{enumerate}
    After exhaustively examining all rank $3$ subgroups of  $\Hol(3_{+}^{1+4})$ and $\Hol(P(\epsilon))$ using \Magma, we conclude that neither group contains semiprimitive rank $3$ subgroups.
    Hence, cases (a) and (b) are excluded.

    Assume that $P$ is the group in case~(c).
    Then $|\bfZ(P)|=|P/\bfZ(P)|=2^n$, see~\cite[page 8]{li2024Proof}.
    By Lemma~\ref{lem:semi}\,(ii) and (iii), we have that $G\lesssim\GL(P/\bfZ(P))\cong \GL_n(2)$ is transitive on non-identities of $P/\bfZ(P)\cong \rmC_2^n$ with order divisible by $2^n(2^n-1)$.
    Then $G$ is non-solvable by~\cite[Corollary 2.12]{huang2025finite}, which is a contradiction.
    Thus, case~(c) is excluded.

    Now, we assume that $P$ is a Sylow $p$-subgroup of $\SU_3(p^n)$ for some prime $p$.
    Remark that $|\bfZ(P)|=p^{n}$ and $|P/\bfZ(P)|=p^{2n}$.
    By Lemma~\ref{lem:semi}, we have that $G\lesssim\GL(P/\bfZ(P))\cong \GL_{2n}(p)$ is transitive on non-identities of $P/\bfZ(P)\cong\rmC_p^{2n}$ with order divisible by $p^{n}(p^{2n}-1)$.
    With the classification of finite solvable $2$-transitive groups~\cite{huppert1957Zweifach} (also see~\cite[Theorem 7.3]{huppert1982Finite}), either $G\lesssim\GammaL_1(p^{2n})$ or $p^{2n}\in\{3^2,5^2,7^2,11^2,23^2,3^4\}$.
    On the one hand, if $G\lesssim\GammaL_1(p^{2n})$, then $|\GammaL_1(p^{2n})|=2n(p^{2n}-1)$ is divisible by $p^n$.
    This yields that $p^{2n}=2^2$ or $2^4$.
    On the other hand, assume that $G\not\lesssim\GammaL_1(p^{2n})$.
    We remark that one can exhaust all primitive groups in small degrees by using the \textbf{PrimitiveGroups} command in \Magma.
    By calculations, we have that if $p^{2n}{:}G$ is a solvable $2$-transitive group and $|G|$ is divisible by $p^n(p^{2n}-1)$, then $p^{2n}=3^2$.
    Thus, we conclude that $P\cong \PSU_3(2)_2$, $\PSU_3(4)_2$, or $\PSU_3(3)_3$.
    With the aid of~\Magma, we obtain that $\Hol(\PSU_3(3)_3)$ has no semiprimitive rank $3$ subgroups.
    Therefore, $X$ is one of the groups in Examples~\ref{ex:Q8} and~\ref{ex:Stab-SU3(4)} by Lemma~\ref{lem:unique}.
\end{proof}

Finally, we will show that $G$ cannot be non-solvable.
For $q=p^n$ with odd prime $p$, the \textit{extraspecial $q$-group} is defined by
\[q^{1+2m}=q^{1+2}_+\circ\cdots\circ q^{1+2}_+=(q^{1+2}_+)^m/C,\]
where $q^{1+2}_+=\SL_3(q)_p$ and $C$ is generated by elements $(X_1,...,X_m)\in\bfZ(q^{1+2}_+)^m$ such that $\prod_{i=1}^m X_i=1$, refer to~\cite[Definition 2.17]{li2025finite}.
We remark that $q^{1+2m}_+$ is a special $p$-group, and hence $q^{1+2m}_+/\bfZ(q^{1+2m}_+)$ is an elementary abelian $p$-group of order $q^{2m}$.
Then $\Hol(q^{1+2m}_+)$ is a rank $3$ group by~\cite[Lemma~5.4]{li2025finite}.

Now we give the following reduction in the following lemma.

\begin{lemma}\label{lem:semi_sl3}
    Assume that $G$ is non-solvable.
    Then $P\cong q^{1+2m}_+/U$ for some $U< \bfZ(q^{1+2m}_+)$.
\end{lemma}
\begin{proof}
    Assume that $P$ is not isomorphic to a quotient of $q^{1+2m}_+/U$.
    Notice that $G\leqs \Aut(P)$ is non-solvable with $3$ orbits acting on $P$.
    By~\cite[Table 1]{li2025finite}, we have that
    \[P=\left\langle M(a,b,\theta)\mid a,b\in\bbF_{p^n}\right\rangle,\]
    where $\theta\in\Aut(\bbF_{p^n})$ has order $3$ and $M(a,b,\theta)$ is a matrix of form
    \[M(a,b,\theta)=\begin{pmatrix}1&a&b\\0&1&a^\theta\\0&0&1\end{pmatrix}.\]
    We also have that $|\bfZ(P)|=|P/\bfZ(P)|=p^n$.
    Lemma~\ref{lem:semi}\,(iv) implies that $G$ is a transitive linear group on $V=\bbF_p^n$ and $\Lambda^2(V)/W\cong \bbF_p^n$ for some $\bbF_pG$-submodule $W$.
    Since $G$ is non-solvable, we have that $G^{(\infty)}\cong \SL_3(p^{n/3})$ by~\cite[Theorem A]{li2025finite}.
    Let $x\in P\setminus\bfZ(P)$, and let $\overline{x}=x\bfZ(P)$.
    Then we have that
    \[|\bfC_{G}(\overline{x}){:}\bfC_G(x)|=|\overline{x}|=|\bfZ(P)|=p^n.\]
    Since the induced linear group of $G^{(\infty)}$ acting on $P/\bfZ(P)\cong\rmC_p^n$ is $\SL_3(p^{n/3})$.
    It follows that $\bfC_{G^{(\infty)}}(\overline{x})\cong p^{2n/3}{:}\SL_2(p^{n/3})$.
    In addition, we have that
    \[\bfC_{G}(\overline{x})/\bfC_{G^{(\infty)}}(\overline{x})\cong G/G^{(\infty)}\lesssim \GammaL_1(p^{n/3}).\]
    Since $|O_p(\bfC_{G^{(\infty)}}(\overline{x}))||\GammaL_1(p^{n/3})_p|<p^n$, we have that $\bfC_{G^{(\infty)}}(\overline{x})/O_p(\bfC_{G^{(\infty)}}(\overline{x}))\cong \SL_2(p^{n/3})$ has a subgroup of index $p^k$ for some positive integer $k$.
    By~\cite[Theorem 1]{guralnick1983Subgroups}, we have that either $\SL_2(p^{n/3})$ is solvable, or $p^{n/3}=5$, $7$ or $11$.
    Thus, it yields that $p^n=3^{3}$, $5^3$, $7^3$ or $11^3$.
    Hence, we have that 
    \[G^{(\infty)}\cong\SL_3(p)\mbox{ and }G/G^{(\infty)}\lesssim\GL_3(p)/\SL_3(p)\cong\rmC_{p-1}.\]
    Note that $G$ is transitive on $P\setminus\bfZ(P)$ and $G^{(\infty)}\cong\SL_3(p)\lesssim\Aut(P)$ is transitive on $P/\bfZ(P)$.
    Then $(G^{(\infty)})_{\overline{x}}$ is transitive on $\overline{x}$ since $G/G^{(\infty)}\lesssim \GL_3(p)/\SL_3(p)\cong\rmC_{p-1}$ has order coprime to $|\overline{x}|=p^3$.
    This yields that $G^{(\infty)}\cong\SL_3(p)$ acts transitively on $P\setminus\bfZ(P)$.
    In other words, there exists a transitive subgroup (acting on $P\setminus\bfZ(P)$) isomorphic to $\SL_3(p)$ of $\Aut(P)$.
    With the aid of \Magma, we have that there is no subgroup isomorphic to $\SL_3(p)$ in $\Aut(P)$ which is transitive on $P\setminus\bfZ(P)$ when $p\in\{3,5,7,11\}$.
    Therefore, there is no such semiprimitive rank $3$ group when $P\not\cong q^{1+2m}_+/U$ for some $U<\bfZ(P)$.
\end{proof}

To complete the proof of Theorem~\ref{thm:semiprim}, we only need to show that $P$ cannot be a quotient of some extraspecial $q$-group for some $q=p^n$ with odd prime $p$.
The following properties of $P$ can be found in~\cite[Section 5]{li2025finite}.

\begin{proposition}\label{prop:extraspecialq}
    Let $p$ be an odd prime, and let $P$ be a quotient of some extraspecial $q$-group with $q=p^n$ such that $\Hol(P)$ has rank $3$.
    We further assume that $q$ is the minimal number among powers of $p$ such that $P$ is a quotient of $q^{1+2m}_+$ for some $m$.
    Then we set $Q=q^{1+2m}_+$ and $P=Q/U$ with $U<\bfZ(Q)$.
    \begin{enumerate}
        \item $\Aut(Q)=K{:}(S{:}L)$, where
        \begin{enumerate}
            \item $K\cong \rmC_p^{2mn}$ is the kernel of $\Aut(Q)$ acting on $Q/\bfZ(Q)$;
            \item $S{:}L\cong \CGammaSp_{2m}(q)$ such that $S\cong\Sp_{2m}(q)$ acts trivially on $\bfZ(Q)$ and $L\cong\GammaL_1(q)$ acts faithfully on $Q/\bfZ(Q)$.
        \end{enumerate}
        \item $\Aut(P)$ is naturally induced by $K{:}(S{:}L_U)\leqs\Aut(Q)$.
    \end{enumerate}
\end{proposition}

In the proof of the next lemma, we will apply some methods of group cohomology.
For an $\bbF G$-module $V$, denote by $\H^{i}(G,V)$ the $i$-th cohomology group of $G$ on $V$.
In particular, $|\H^{0}(G,V)|$ equals the number of fixed points of $G$ on $V$, and $|\H^{1}(G,V)|$ equals to the number of conjugacy classes of complements of $V$ in $V{:}G$.
See~\cite{aschbacher1984applications} for more detailed applications of cohomology groups.
\begin{lemma}\label{lem:semi_extra}
    Using all notations defined in Proposition~$\ref{prop:extraspecialq}$, we assume that $G\leqs \Aut(P)$ is non-solvable and $P{:}G$ is semiprimitive.
    Then $P{:}G$ has rank more than $3$.
\end{lemma}
\begin{proof}
    Suppose that $P{:}G$ has rank $3$.
    Then $G$ acts faithfully on $P/\bfZ(P)$, and is transitive on non-identities of both $\bfZ(P)$ and $P/\bfZ(P)$.
    We will present our proof in four steps.

    \textbf{Step 1.}
    First, we prove that there exists a subgroup $H$ isomorphic to $S{:}L\cong \CGammaSp_{2m}(q)$ in $\Aut(Q)=K{:}(S{:}L)$ such that $\bfC_H(y)=\bfC_H(\overline{y})$ for some $y\in Q\setminus \bfZ(Q)$, where $\overline{y}=y\bfZ(Q)$.

    Let $u_1,...,u_{m+1},v_1,...,v_{m+1}$ be a symplectic basis of $\bbF_{q}^{2m+2}$ such that $(u_i,v_i)$ is a hyperbolic pair for each $i=1,...,m+1$, refer to~\cite[Section 3.4.4]{wilson2009finite}.
    By~\cite[Corollary 4.5]{li2025finite}, we have that $Q\cong O_p(T_{\langle u_1\rangle})$, where $T=\Sp_{2m+2}(q)$.
    We identify $Q$ with $O_p(T_{\langle u_1\rangle})$.
    Set $W=\langle u_1,v_1\rangle^\perp$, and $H_0=\Sp(W)\cong \Sp_{2m}(q)$.
    Then $H_0$ acts faithfully and transitively on non-identities of $Q/\bfZ(Q)$.
    Let $y\in Q\setminus\bfZ(Q)$ (see~\cite[Section 3.5.3]{wilson2009finite}) be the map fixing all basis vectors $u_i$ and $v_i$ except:
    \[y(v_1)=v_1+u_2\mbox{ and }y(v_2)=v_2+u_1.\]
    Let $h\in(H_0)_{u_2}$.
    Then
    \[\begin{aligned}
        hy(v_1)&=h(v_1+u_2)=v_1+u_2=yh(v_1),\\
        hy(u_1)&=u_1=yh(u_1)\mbox{ and }hy(u_2)=u_2=yh(u_2).
    \end{aligned}\]
    Note that $h(v_2)=v_2+ku_2+w$ for some integer $k$, where $w\in\langle v_1,v_2,u_1,u_2\rangle^\perp$.
    Then
    \[hy(v_2)=h(v_2+u_1)=v_2+ku_2+w+u_1=y(v_2+k(u_2)+w)=yh(v_2).\]
    For $i\geqs 3$, we have that $h(v_i)=b_iu_2+w_i$ and $h(u_i)=c_i u_2+w_i'$ for integers $b_i,c_i$ with $w_i,w_i'\in\langle u_1,u_2,v_1,v_2\rangle^\perp$.
    Hence, we have that
    \[\begin{aligned}
        hy(v_i)&= b_iu_2+w_i=y(b_iu_2+w_i)=yh(v_i),\\
        hy(u_i)&= c_i u_2+w_i'=y(c_i u_2+w_i')=yh(u_i).
    \end{aligned}\]
    Then $h\in \bfC_{H_0}(y)$, and hence $(H_0)_{u_2}\leqs \bfC_{H_0}(y)$.
    Note that
    \[|H_0{:}\bfC_{H_0}(\overline{y})|\leqs|H_0{:}\bfC_{H_0}(y)|\leqs |H_0{:}(H_0)_{u_2}|=q^{2m}-1.\]
    It follows that $\bfC_{H_0}(\overline{y})=\bfC_{H_0}(y)=(H_0)_{u_2}$.
    Let $\mu$ be a primitive element in $\bbF_q^\times$.
    We set two maps $\delta,\phi\in\CGammaSp_{2m+2}(q)$ such that
    \[
    \begin{array}{lll}
        &\phi(\mu u_i)=\mu^p u_i, &\phi(\mu v_i)=\mu^p v_i\mbox{ and }\\
        &\delta(u_i)=
        \left\{\begin{aligned}
            &u_i,\mbox{ if }i\neq 2,\\
            &\mu u_2,\mbox{ if }i=2;
        \end{aligned}\right.,
        &\delta(v_i)=
        \left\{\begin{aligned}
            &\mu v_i,\mbox{ if }i\neq 2,\\
            &v_2,\mbox{ if }i=2;
        \end{aligned}\right.
    \end{array}
    \]
    It is easy to check that $H=\langle H_0,\delta,\phi\rangle\cong\CGammaSp_{2m}(q)$ and both of $\phi,\delta$ commute with $y$.
    Hence, we have that $\bfC_H(y)=(H_0)_{u_2}{:}\langle\delta,\phi\rangle$.
    It follows that $\bfC_H(y)$ has index $q^{2m}-1$ in $H$, and then $\bfC_H(y)=\bfC_H(\overline{y})$.

    \textbf{Step 2.} 
    We show that if $P{:}G$ has rank $3$, then $P=p^{1+2d}_+$ for some $d$.

    Suppose that $P$ is not an extraspecial $p$-group.
    Then $\bfZ(P)$ has order larger than $p$, that is, $\dim\bfZ(P)>1$.
    By Lemma~\ref{lem:semi}\,(iv) and Proposition~\ref{prop:extraspecialq}\,(ii), we have that $G^{(\infty)}$ is isomorphic to $\Sp_{2m'}(p^{n'})$ such that $m'n'=mn$ and $n'>1$.

    Let $\overline{K}$ be the group induced by $K$ acting on $P$, and let $\overline{H_U}$ be the induced permutation group of $H_U$ acting on $P$.
    Note that $\overline{K}$ is the kernel of $\Aut(P)$ acting on $P/\bfZ(P)$. 
    Then $\Aut(P)=\overline{K}{:}\overline{H_U}$ by~\cite[Theorem~B]{li2025finite}, and $G\lesssim H_U\lesssim \Sp_{2m}(q){:}\GammaL_1(q)$ is faithful and transitive on $P/\bfZ(P)\cong\rmC_p^{2mn}$ with odd prime $p$. 
    We remark that $\overline{K}$ can be identified as the $\bbF_p\overline{H_U}$-module
    \[\overline{K}\cong \Hom_{\bbF_p}(P/\bfZ(P),\bfZ(P))\cong (P/\bfZ(P))\otimes \bfZ(P)^*.\]
    Since $G^{(\infty)}\leqs \overline{H^{(\infty)}}\cong\Sp_{2m}(q)$ acts trivially on $\bfZ(P)$ by Proposition~\ref{prop:extraspecialq}, the corresponding $\bbF_pG^{(\infty)}$-module of $\overline{K}$ is isomorphic to a direct sum of $t$-copies of $P/\bfZ(P)$, where $t=\dim \bfZ(P)$.
    Hence, we have that $\H^0(G^{(\infty)},\overline{K})=0$.
    Then, by~\cite[2.7]{aschbacher1984applications}, we obtain that
    \[|\H^1(G,\overline{K})|\leqs |\H^1(G^{(\infty)},\overline{K})|=|\H^1(G^{(\infty)},P/\bfZ(P))|^t.\]
    Since $p^{n'}>3$, \cite[Theorem 1.2.2]{group2013First} yields that $|\H^1(G^{(\infty)},P/\bfZ(P))|=1$, and hence $|\H^1(G,\overline{K})|=1$.
    Recall that $\overline{H_U}$ is a complement of $\overline{K}$ in $\Aut(P)$ and $G\cap\overline{K}=1$.
    Then there exists $\varphi\in \overline{K}$ such that $G^\varphi\leqs \overline{H_U}$. 
    Let $x=(yU)^{\varphi^{-1}}$, where $yU$ is the image of $y$ in $P=Q/U$.
    It follows that $\bfC_{G}(x)=\bfC_G(\overline{x})$.
    As $P{:}G$ has rank $3$, $P\setminus\bfZ(P)$ is an orbit of $G$. 
    However, we have that 
    \[q^{2m}-1=|G{:}\bfC_G(\overline{x})|=|G{:}\bfC_G(x)|<(q^{2m}-1)|\bfZ(P)|.\]
    This is a contradiction.

    \textbf{Step 3.}
    We assume from now that $P=Q=p^{1+2m}_+$.
    Then $p=3$ and $G^{(\infty)}=\Sp_{2m}(3)$.

    The proof of this step is similar to the proof of Step 2.
    Recall that $\Aut(P)=K{:}H=K{:}(S{:}L)$ with $H=S{:}L\cong\CGammaSp_{2m}(p)$ as given in Step 1.
    Then $K$ is naturally an $\bbF_pH$-module such that 
    \[K\cong \Hom_{\bbF_p}(P/\bfZ(P),\bfZ(P))\cong (P/\bfZ(P))\otimes \bfZ(P)^*.\]
    On the one hand, suppose that $G^{(\infty)}$ is not isomorphic to a symplectic group.
    Then, by Hering's classification of non-solvable affine $2$-transitive groups~\cite{hering1985Transitive}, we have that $\bfZ(G^{(\infty)})\neq 1$ is a $p'$-group with $\H^0(\bfZ(G^{(\infty)},P/\bfZ(P)))=0$.
    Thus, we have that
    \[|\H^1(G,K)|\leqs |\H^1(G^{(\infty)},P/\bfZ(P))|^t\leqs |\H^1(\bfZ(G^{(\infty)}),P/\bfZ(P))|^t=1.\]
    On the other hand, suppose that $G^{(\infty)}\cong\Sp_{2m}(p)$ with $p>3$.
    Then \cite[Theorem 1.2.2]{group2013First} implies that $|\H^1(G^{(\infty)},P/\bfZ(P))|=1$, and hence $|\H^1(G,K)|=1$.
    Thus, we have that $G$ is conjugate to a subgroup of $H$ in $\Aut(P)$.
    Then the rank of $P{:}G$ is larger than that of $P{:}H$.
    Note that $\bfC_H(y)=\bfC_H(\overline{y})$ for some $y\in P\setminus\bfZ(P)$ proved in Step 1.
    Hence, we have that
    \[p^{2m}-1=|H{:}\bfC_H(\overline{y})|=|G{:}\bfC_G(y)|<(p^{2m}-1)|\bfZ(P)|.\]
    By Lemma~\ref{lem:semi}\,(iii), we have that $P{:}H$ has rank more than $3$.
    
    \textbf{Step 4.}
    Finally, we complete the proof by assuming that $P=Q=3^{1+2m}_+$ and $G^{(\infty)}=\Sp_{2m}(3)$.

    Since $G$ is non-solvable, we have that $m\geqs 2$. 
    If $P{:}G$ has rank $3$, then $\bfC_G(x)$ has index $3$ in $\bfC_G(\overline{x})$.
    Note that $\bfC_G(\overline{x})/\bfC_{G^{(\infty)}}(\overline{x})$ has order dividing $|G/G^{(\infty)}|$, and $G/G^{(\infty)}$ has order $1$ or $2$.
    This yields that the group $Y=\bfC_{G^{(\infty)}}(\overline{x})$ has a subgroup of index $3$.
    Remark that $Y\cong 3^{1+2(m-1)}_+{:}\Sp_{2m-2}(3)$ and $Y$ acts irreducibly on $O_p(Y)/\bfZ(O_p(Y))\cong\rmC_3^{2m-2}$.
    Then we have that $\Sp_{2m-2}(3)$ has a subgroup of index $3$.
    By~\cite{guralnick1983Subgroups}, we obtain that $\Sp_{2m-2}(3)$ is solvable, and hence $m=2$.
    It follows that $P=Q=3^{1+4}_+$.
    With the aid of \Magma, there is no semiprimitive group of rank $3$ of the form $3^{1+4}_+{:}G$ for any $G\leqslant\Hol(3^{1+4}_+)$.
    The proof is complete.
\end{proof}

We conclude the proof of Theorem~\ref{thm:semiprim} below.

\begin{proof}[Proof of Theorem~$\ref{thm:semiprim}$]
    By~\cite[Theorem 1]{huang2025finite}, either $X$ is one of groups in case~(i) or first two rows in Table~\ref{tab:semi}; or $P\lhd X\leqs\Hol(P)$ for some special $p$-group $P$.

    Assume that $P\lhd X\leqs\Hol(P)$ for some non-abelian special $p$-group $P$.
    Let $G$ be the stabilizer of $X$.
    If $G$ is solvable, then Lemma~\ref{lem:semi-nonsol} proves that $G$ is one of groups in Examples~\ref{ex:Q8} and~\ref{ex:Stab-SU3(4)}, which are groups in the third and the fourth rows of Table~\ref{tab:semi}.
    If $G$ is non-solvable, then Lemmas~\ref{lem:semi_sl3} and~\ref{lem:semi_extra} show that $P{:}G$ cannot be rank $3$.

    Assume that $X$ has an abelian regular normal subgroup.
    Then Lemma~\ref{lem:semiaffine} deduces that $X$ is as given in Example~\ref{exam:two}, which lies in the last row of Table~\ref{tab:semi}.
    Therefore, the proof is complete.
\end{proof}

\vskip0.2in

\noindent{\bf Data availability}
\vskip0.1in
No data was used for the research described in the article.

\vskip0.2in

\noindent{\bf Declaration of competing interest}
\vskip0.1in
We have no conflict of interest.


\end{document}